\documentclass[11pt,reqno]{amsart}
\usepackage{amssymb,latexsym}
\usepackage{graphicx}
\usepackage{color} 
\usepackage{url}		
\calclayout

\theoremstyle{plain}
\numberwithin{equation}{section}
\newtheorem{thm}{Theorem}[section]
\newtheorem{theorem}[thm]{Theorem}
\newtheorem{lemma}[thm]{Lemma}

\newtheorem{proposition}[thm]{Proposition}
\newtheorem{corollary}[thm]{Corollary}
\newtheorem{remark}[thm]{Remark}

\begin{document}
\setcounter{page}{1}

\title[Combinatorial sums derived from properties of Legendre polynomials]
{Combinatorial sums derived from properties of Legendre polynomials}
\author{Michel Bataille}
\address{Independent Researcher, 76520 Franqueville-Saint-Pierre, France}
\email{michelbataille@wanadoo.fr}
\medskip
\author{Robert Frontczak}
\address{Independent Researcher, 72764 Reutlingen, Germany}
\email{robert.frontczak@web.de}

\begin{abstract}
From an identity connecting a combinatorial sum and Legendre polynomials, we derive closed forms for a number of combinatorial sums. Some of them are obtained \textit{via} results about the integrals of functions associated with Legendre polynomials. \\
\\
\bigskip
{\sc Key words}: Sum, (central) binomial coefficient, Legendre polynomial, Gamma function. 
\\
\bigskip
{\sc MSC 2020}: 05A10, 11B65, 11B83.  
\end{abstract}


\maketitle

\section{Introduction}

As an introduction and for the convenience of the reader, we start by briefly reviewing some basic properties of the Legendre polynomials.
This important class of polynomials is named after A.-M. Legendre, who discovered them around 1780. \\
First, recall that for a complex number $x$ these polynomials $P_n(x),\ (n=0,1,\ldots)$ are defined by
\[P_n(x) = \frac{1}{n! 2^n} \frac{d^n}{dx^n} \left((x^2-1)^n\right),\]
which is called Rodrigues’ formula for $P_n(x)$. Applying Leibniz's formula to the product $(x+1)^n (x-1)^n$, we get their equivalent form
\begin{equation}\label{intro_a} 
P_n(x) = \frac{1}{2^n} \sum_{k=0}^n \binom{n}{k}^2 (x+1)^{n-k} (x-1)^k.
\end{equation}
Thus, $P_0(x)=1,\ P_1(x)=x,\ P_2(x)=\frac{3x^2-1}{2},\ldots$.\\
The following important formula (generating series) is well-known:
\begin{equation}\label{intro_b}
\frac{1}{(z^2-2xz+1)^{1/2}} = \sum_{n=0}^{\infty} P_n(x) z^n.
\end{equation}
We immediately deduce that for all $n\ge 0$ we have $P_{2n+1}(0)=0$  and $P_{2n}(0)=\binom{-1/2}{2n}=(-1)^n 2^{-2n}\binom{2n}{n}$ as well as $P_n(1)=1$. Note also that $P_n(-x)=(-1)^n P_n(x)$. \\
By differentiation with respect to $z$ followed by a multiplication by $1-2xz+z^2$, we derive the recursion
\begin{equation}\label{intro_c}
(n+1)P_{n+1}(x) - (2n+1) x P_n(x) + nP_{n-1}(x) = 0 \quad (n\ge 1).
\end{equation}
We also recall the orthogonality property of the Legendre polynomials, that is, 
$$\int_{-1}^1 P_n(x)P_m(x)\,dx = 0\quad\text{for}\,\, n\neq m.$$ 
In particular with $m=0$, we obtain that 
$$\int_{-1}^1 P_n(x)\,dx = 0 \quad\text{for}\,\, n\ge 1.$$ \\
Although classical, Legendre polynomials are still the subject of mathematical research. Papers covering different aspects of the topic 
include Klemm and Larsen \cite{Klemm}, Wan and Zudilin \cite{Wan}, Diekemaa and Koornwinder \cite{Diekemaa}, Guo \cite{Guo}, 
Chu and Campbell \cite{Chu}, and Aloui \cite{Aloui}, to mention a few.
The On-Line Encyclopedia of Integer Sequences (OEIS) \cite{OEIS} contains some number sequences associated with Legendre polynomials:
A001801, A008316, A110129 or A330203. \\
In this article, we first use a coefficient extraction formula to link a class of combinatorial sums to Legendre polynomials.
From this connection we deduce a range of combinatorial sums via integrals of certain functions associated with Legendre polynomials.

\section{A first result}

Let $c_n, n\geq 0,$ be a real or complex sequence of numbers and let $F(z)$ be its ordinary generating function, i.e.,
$$F(z)=\sum_{k=0}^\infty c_k\,z^k.$$
Let $[z^n]F(z)$ denote the coefficient in $F(z)$ belonging to $z^n$. We state the following coefficient extraction identity 
as a lemma:
\begin{lemma}\label{main_lem}
For all $n\geq 1$ we have
\begin{equation}
\sum_{k=0}^n (-1)^{n-k} \frac{2n}{n+k} \binom{n+k}{2k} c_k = [z^n] \frac{1-z}{1+z}F\left (\frac{z}{(1+z)^2}\right ).
\end{equation}
\end{lemma}
\begin{proof}
From
\[\frac{1-z}{1+z}F\left (\frac{z}{(1+z)^2}\right )=\sum_{k=0}^{\infty}\frac{c_k z^k}{(1+z)^{2k+1}}-\sum_{k=0}^{\infty}\frac{c_k z^{k+1}}{(1+z)^{2k+1}}\]
and 
\[\frac{1}{(1+z)^{2k+1}} = \sum_{m=0}^{\infty} (-1)^m \binom{2k+m}{2k} z^m,\]
we deduce that
\begin{align*}
[z^n] \frac{1-z}{1+z}F\left (\frac{z}{(1+z)^2}\right ) &= \sum_{k=0}^n c_k(-1)^{n-k} \left(\binom{2k+n-k}{2k}+\binom{2k+n-k-1}{2k} \right) \\
&= \sum_{k=0}^n c_k (-1)^{n-k} \frac{(n+k-1)!}{(2k)!(n-k-1)!}\cdot\frac{2n}{n-k} \\
&= \sum_{k=0}^n (-1)^{n-k} \frac{2n}{n+k} \binom{n+k}{2k} c_k. 
\end{align*}
\end{proof}

Now, our first main result follows.

\begin{theorem}\label{main_thm}
Let $x\in\mathbb{C}$. Then for all $n\geq 1$ we have
\begin{equation}
\sum_{k=0}^n \binom{n+k}{2k} \binom{2k}{k} \frac{2^{-2k}}{n+k} x^k 
= \frac{(-1)^n}{2n} \left ( P_n\left (-\frac{x+2}{2}\right ) - P_{n-1}\left (-\frac{x+2}{2}\right )\right ),
\end{equation}
where $P_n(x)$ is the $n$th Legendre polynomial.
\end{theorem}
\begin{proof}
We apply Lemma \ref{main_lem}. Let $c_k=\binom{-1/2}{k}x^k=\binom{2k}{k}2^{-2k} (-1)^k x^k.$ Then
$$F(z) = \frac{1}{(1+xz)^{1/2}},$$
and
$$\frac{1-z}{1+z} F\left (\frac{z}{(1+z)^2}\right ) = \frac{1-z}{(z^2+(2+x)z+1)^{1/2}}.$$
From \eqref{intro_b}, we obtain
$$\frac{1}{(z^2+(2+x)z+1)^{1/2}} = \sum_{k=0}^\infty P_k\left (-\frac{x+2}{2}\right )\,z^k, $$
and the proof is complete.
\end{proof}

\begin{remark}
We note that
\[\binom{n+k}{k}\binom{n}{k}=\binom{n+k}{n}\binom{n}{n-k}=\binom{n+k}{n-k}\binom{2k}{k}=\binom{n+k}{2k}\binom{2k}{k}.\]  
\end{remark}

\begin{corollary}
For all $n\geq 1$
\begin{equation}
\sum_{k=0}^n \binom{n+k}{2k} \binom{2k}{k} \frac{2^{-2k}}{n+k} = \frac{(-1)^n}{2n} ( S_n - S_{n-1}),
\end{equation}
and
\begin{equation}
\sum_{k=0}^n \binom{n+k}{2k} \binom{2k}{k} (-1)^k \frac{2^{-2k}}{n+k} = \frac{(-1)^n}{2n} ( Q_n - Q_{n-1}),
\end{equation}
where
$$S_n=\left (-\frac{5}{4}\right )^n \sum_{k=0}^n \binom{n}{k}^2 5^{-k}$$
and
$$Q_n=\left (-\frac{3}{4}\right )^n \sum_{k=0}^n \binom{n}{k}^2 (-1)^k 3^{-k}.$$
\end{corollary}
\begin{proof}
From \eqref{intro_a}, $S_n$ and $Q_n$ are the values for $P_n(-3/2)$ and $P_n(-1/2)$, respectively.
\end{proof}

\begin{corollary}
For $n\geq 1$
\begin{equation}\label{xxx_id}
\sum_{k=0}^{n} \binom{n+k}{2k} \binom{2k}{k} \frac{(-1)^k}{n+k} = 0.
\end{equation}
\end{corollary}
\begin{proof}
Use Theorem \ref{main_thm} and the fact that $P_{n}(1) = 1$ for all $n\geq 0$.
\end{proof}

\begin{corollary}
For $n\geq 1$
\begin{equation}
\sum_{k=0}^{2n} \binom{2n+k}{2k} \binom{2k}{k} \frac{(-1)^k 2^{-k}}{2n+k} = \frac{(-1)^n}{4n} 2^{-2n} \binom{2n}{n},
\end{equation}
and $n\geq 0$
\begin{equation}
\sum_{k=0}^{2n+1} \binom{2n+1+k}{2k} \binom{2k}{k}  \frac{(-1)^k 2^{-k}}{2n+1+k} = \frac{(-1)^n}{2(2n+1)} 2^{-2n} \binom{2n}{n}.
\end{equation}
\end{corollary}
\begin{proof}
Use the values
$$P_{2n+1}(0) = 0 \qquad\text{and}\qquad P_{2n}(0) = (-1)^n 2^{-2n} \binom{2n}{n}.$$
\end{proof}

\section{Some combinatorial identities derived from integration}

\begin{proposition}
We have
\begin{equation}\label{id_int1}
\sum_{k=0}^{n} \binom{n+k}{2k} \binom{2k}{k} \frac{(-1)^k}{(n+k)(k+1)} = 
\begin{cases}
1/2, & n=1 \\ 
0, & n\geq 2.
\end{cases} 
\end{equation}
\end{proposition}
\begin{proof}
Under the transformation $x\mapsto -4x$, the main identity in Theorem \ref{main_thm} becomes
\begin{equation}\label{main_id_mod}
\sum_{k=0}^n \binom{n+k}{2k} \binom{2k}{k} \frac{(-1)^k}{n+k} x^k = \frac{(-1)^n}{2n} \left ( P_n(2x-1) - P_{n-1}(2x-1)\right ).
\end{equation}
The result follows by integrating both sides from 0 to 1 and noting that for $n\ge 1$
$$\int_0^1 P_n(2x-1) dx = \frac{1}{2}\int_{-1}^1 P_n(x) dx=0.$$
\end{proof}

\begin{corollary}
We have
\begin{equation}
\sum_{k=0}^{n} \binom{n+k}{2k} \binom{2k}{k} \frac{(-1)^k}{(n+k)}\frac{k}{k+1} = 
\begin{cases}
-1/2, & n=1 \\ 
0, & n\geq 2.
\end{cases} 
\end{equation}
\end{corollary}
\begin{proof}
Combine \eqref{id_int1} with \eqref{xxx_id}.
\end{proof}

\begin{theorem}\label{main_int_thm1}
Let $\mu$ be a complex number with $\Re(\mu)> 0$. Then
\begin{equation}\label{main_int_id1}
\sum_{k=0}^{n} \binom{n+k}{2k} \binom{2k}{k} \frac{(-1)^k}{(n+k)(2k+2\mu)} 
= \frac{(-1)^n}{4n} \left (\frac{\Gamma^2(\mu)}{\Gamma(\mu+n+1)\Gamma(\mu-n)} - \frac{\Gamma^2(\mu)}{\Gamma(\mu+n)\Gamma(\mu+1-n)} \right ),
\end{equation}
where $\Gamma(z)$ denotes the famous Gamma function defined by
\begin{equation*}
\Gamma(z) = \int_0^\infty e^{-t} t^{z-1} dt \qquad (\Re(z)>0).
\end{equation*}
\end{theorem}
\begin{proof}
We shall use the formula
$$\int_0^1 x^{2\mu -1}P_n(2x^2-1) dx = \frac{\Gamma^2(\mu)}{2 \Gamma(\mu+n+1)\Gamma(\mu-n)}, \quad \Re(\mu)>0$$
directly extracted from Gradshteyn and Ryzhik \cite[Equation (7.233)]{GrRy07} or easily proved by induction using \eqref{intro_c}.
Hence, working with \eqref{main_id_mod}, we get
\begin{align*}
&\sum_{k=0}^{n} \binom{n+k}{2k} \binom{2k}{k} \frac{(-1)^k}{(n+k)(2k+2\mu)} = \sum_{k=0}^{n} \binom{n+k}{2k} \binom{2k}{k} \frac{(-1)^k}{(n+k)} \int_0^1 x^{2k+2\mu-1} dx \\
&\qquad\qquad = \frac{(-1)^n}{2n} \left (\int_0^1 x^{2\mu -1}P_n(2x^2-1) dx - \int_0^1 x^{2\mu -1}P_{n-1}(2x^2-1) dx \right ),
\end{align*}
and the proof is finished.
\end{proof}

\begin{corollary}\label{cor_int}
For $n\geq 1$ we have
\begin{equation}\label{id_mu1}
\sum_{k=0}^{n} \binom{n+k}{2k} \binom{2k}{k} \frac{(-1)^k}{(n+k)(2k+1)} = \frac{2}{(2n-1)(2n+1)},
\end{equation}
\begin{equation}\label{id_mu2}
\sum_{k=0}^{n} \binom{n+k}{2k} \binom{2k}{k} \frac{(-1)^k}{(n+k)(2k+3)} = - \frac{2}{(2n-3)(2n-1)(2n+1)(2n+3)},
\end{equation}
\begin{equation}\label{id_mu3}
\sum_{k=0}^{n} \binom{n+k}{2k} \binom{2k}{k} \frac{(-1)^k}{(n+k)^2} = \frac{(-1)^{n+1}}{n^2 \binom{2n}{n}},
\end{equation}
and
\begin{equation}\label{id_mu4}
\sum_{k=0}^{n} \binom{n+k}{2k} \binom{2k}{k} \frac{(-1)^k}{(n+k)(n+1+k)} = \frac{(-1)^{n+1}}{(2n+1) \binom{2n}{n}}.
\end{equation}
\end{corollary}
\begin{proof}
These identities are special cases of Theorem \ref{main_int_thm1} for $\mu=1/2$, $\mu=3/2$, $\mu=n$, and $\mu=n+1$, respectively.
For instance, using the classical properties of the Gamma function (see \cite{Srivastava})
$$\Gamma(1/2)=\sqrt{\pi}, \quad \Gamma(1+z)=z\Gamma(z),\quad \text{and}\quad \Gamma(z)\Gamma(1-z)=\frac{\pi}{\sin(\pi z)},$$
we calculate
\begin{align*}
\sum_{k=0}^{n} \binom{n+k}{2k} \binom{2k}{k} \frac{(-1)^k}{(n+k)(2k+1)} &= \frac{\pi}{\Gamma(1/2+n)\Gamma(1/2-n)}\frac{(-1)^n}{2n}\left (\frac{1}{2n+1}+\frac{1}{2n-1}\right ) \\
&= \frac{1}{2n}\left (\frac{1}{2n+1}+\frac{1}{2n-1}\right ).
\end{align*}
The case $\mu=3/2$ is similar, noting that $\Gamma(3/2)=\sqrt{\pi}/2$. Setting $\mu=n$ in Theorem \ref{main_int_thm1} produces
\begin{align*}
\sum_{k=0}^{n} \binom{n+k}{2k} \binom{2k}{k} \frac{(-1)^k}{(n+k)^2} 
&= \frac{(-1)^n}{2n} \left (\frac{\Gamma^2(n)}{\Gamma(2n+1)\Gamma(0)} - \frac{\Gamma^2(n)}{\Gamma(2n)\Gamma(1)} \right ) \\
&= \frac{(-1)^{n+1}}{2n} \frac{(n-1)!^2}{(2n-1)!}.
\end{align*}
It is worth remarking that in the last step we used the fact that since $\Gamma(z)$ has simple poles at $z=0,-1,-2,\ldots$, 
its reciprocal $1/\Gamma(0)$ is a simple zero and the first term vanishes. Finally, with $\mu=n+1$ in Theorem \ref{main_int_thm1} we get 
\eqref{id_mu4} after some simple manipulations.
\end{proof}

\begin{corollary}
For $n\geq 1$ we have
\begin{equation}
\sum_{k=0}^{n} \binom{n+k}{2k} \binom{2k}{k} \frac{(-1)^{k+1}}{(n+k)}\frac{k}{2k+1} = \frac{1}{(2n-1)(2n+1)},
\end{equation}
and
\begin{equation}
\sum_{k=0}^{n} \binom{n+k}{2k} \binom{2k}{k} \frac{(-1)^{k}}{(n+k)}\frac{k+1}{2k+3} = \frac{1}{(2n-3)(2n-1)(2n+1)(2n+3)}.
\end{equation}
\end{corollary}
\begin{proof}
Combine the first two identities in Corollary \ref{cor_int} with identity \eqref{xxx_id}.
\end{proof}

\begin{corollary}
For $n\geq 1$ we have
\begin{equation}
\sum_{k=0}^{n} \binom{n+k}{2k} \binom{2k}{k} \frac{(-1)^{k}}{(n+k)^2 (n+1+k)} = \frac{(-1)^n (n^2-2n-1)}{n^2 (2n+1)\binom{2n}{n}}.
\end{equation}
\end{corollary}
\begin{proof}
Combine the last two identities in Corollary \ref{cor_int}.
\end{proof}

\begin{theorem}\label{main_int_thm2}
Let $m\geq 0$ be an integer such that $m<n-1$. Then
\begin{equation}
\sum_{k=0}^{n} \binom{n+k}{2k} \binom{2k}{k} \frac{(-1)^k}{(n+k)(m+k+1)^2} = \frac{(-1)^m (m!)^2 (n-m-2)!}{(n+m+1)!} 
.
\end{equation}
\end{theorem}
\begin{proof}
Work with \eqref{main_id_mod} again, multiply both sides by $x^m \ln(1/x)$ and integrate from 0 to 1. This yields
\begin{align*}
&\sum_{k=0}^{n} \binom{n+k}{2k} \binom{2k}{k} \frac{(-1)^k}{(n+k)} \int_0^1 x^{k+m}\ln(1/x) dx \\
&\qquad\qquad = \frac{(-1)^m}{2n} \left ( \int_0^1 x^m\ln(1/x) (P_n(2x-1)-P_{n-1}(2x-1)) dx \right ).
\end{align*}
The left-hand side is quickly calculated using the well-known
$$\int_0^1 x^{p}\ln(1/x) dx = \frac{1}{(p+1)^2}$$
for any nonegative integer $p$. \\
To evaluate the integrals on the right, we make use of Gautschi's formula \cite{Gautschi} which states that for $n>m$
$$\int_0^1 x^m \ln(1/x) P_n(2x-1) dx = (-1)^{n-m} (m!)^2 \frac{(n-m-1)!}{(n+m+1)!}.$$
\end{proof}

\begin{corollary}
We have
\begin{equation}
\sum_{k=0}^{n} \binom{n+k}{2k} \binom{2k}{k} \frac{(-1)^k}{(n+k)(k+1)^2} = 
\begin{cases}
3/4, & n=1 \\ 
\frac{1}{(n-1)n(n+1)}, & n\geq 2,
\end{cases} 
\end{equation}
and
\begin{equation}
\sum_{k=0}^{n} \binom{n+k}{2k} \binom{2k}{k} \frac{(-1)^{k+1}}{(n+k)(k+2)^2} = 
\begin{cases}
-5/36, & n=1 \\ 
1/288, & n=2 \\
\frac{1}{(n-2)(n-1)n(n+1)(n+2)}, & n\geq 3.
\end{cases} 
\end{equation}
\end{corollary}
\begin{proof}
Set $m=0$ and $m=1$ in Theorem \ref{main_int_thm2}, respectively.
\end{proof}

\begin{corollary}
For $n\geq 2$ we have
\begin{equation}
\sum_{k=0}^{n} \binom{n+k}{2k} \binom{2k}{k} \frac{(-1)^k}{(n+k)(n-1+k)^2} = (-1)^n \frac{2}{(n-1)^2 n \binom{2n}{n}}.
\end{equation}
\end{corollary}
\begin{proof}
Set $m=n-2$ in Theorem \ref{main_int_thm2} and use 
$$\frac{((n-2)!)^2}{(2n-1)!} = \frac{2}{(n-1)^2 n \binom{2n}{n}}.$$
\end{proof}

We conclude with the following theorem dealing with the square of central binomial coefficients.

\begin{theorem}\label{arcsin_thm}
For $n\geq 1$ we have
\begin{equation}
\sum_{k=0}^{2n} \binom{2n+k}{2k} \binom{2k}{k}^2 2^{-2k} (2k+1) \frac{(-1)^k}{(2n+k)(k+1)^2} = \frac{2^{-4n}}{n^3} \binom{2(n-1)}{n-1}^2
\end{equation}
and
\begin{equation}
\sum_{k=0}^{2n+1} \binom{2n+1+k}{2k} \binom{2k}{k}^2 2^{-2k} (2k+1) \frac{(-1)^k}{(2n+1+k)(k+1)^2} 
= \frac{2^{-(4n+3)}}{(2n+1)(n+1)^2} \binom{2n}{n}^2.
\end{equation}
\end{theorem}
\begin{proof}
The two results are readily checked when $n=1$, so we suppose that $n\ge 2$.\\  
Working with \eqref{main_id_mod} again we see that
$$\sum_{k=0}^n \binom{n+k}{2k} \binom{2k}{k} \frac{(-1)^k}{n+k} 2^{-k} (1+x)^k = \frac{(-1)^n}{2n} \left ( P_n(x) - P_{n-1}(x)\right ),$$
which after multiplying both sides with $\arcsin(x)$ and integrating from $-1$ to 1 becomes
$$\sum_{k=0}^n \binom{n+k}{2k} \binom{2k}{k} \frac{(-1)^k}{n+k} 2^{-k} \int_{-1}^1 (1+x)^k\arcsin(x)\,dx 
= \frac{(-1)^n}{2n} \left ( I_n - I_{n-1}\right ),$$
with
$$I_n = \int_{-1}^1 P_n(x) \arcsin(x)\,dx.$$
We begin by evaluating the integral on the left-hand side. As $\arcsin(x)$ is an odd function it is readily seen that
\begin{equation*}
\int_{-1}^1 x^k\arcsin(x)\,dx = \begin{cases}
0, & k \,\,\text{even} \\ 
 &  \\
2 \int_{0}^1 x^k\arcsin(x)\,dx, & k \,\,\text{odd}
\end{cases} 
\end{equation*}
and hence
$$\int_{-1}^1 (1+x)^k\arcsin(x)\,dx = 2 \sum_{j=1}^{\lfloor (k+1)/2 \rfloor } \binom{k}{2j-1} \int_{0}^1 x^{2j-1} \arcsin(x)\,dx.$$
The next step is to use the integral
$$\int_{0}^1 x^{2n-1} \arcsin(x)\,dx = \frac{\pi}{4n}\left ( 1 - \frac{1}{2^{2n}}\binom{2n}{n} \right ),$$
which follows from
\begin{align*}
\int_0^1 x^{2n-1} \arcsin(x)\,dx &=\int_0^{\pi/2}(\sin u)^{2n-1}u\cos u\,du\\
&=\left[u\cdot\frac{(\sin u)^{2n}}{2n}\right]_0^{\pi/2}-\frac{1}{2n}\int_0^{\pi/2}(\sin u)^{2n}\,du\\
&=\frac{\pi}{2}\cdot\frac{1}{2n}-\frac{1}{2n}\cdot\frac{\pi}{2}\cdot \frac{1}{2^{2n}}\binom{2n}{n}.
\end{align*}
This yields
\begin{equation*}
\int_{-1}^1 (1+x)^k\arcsin(x)\,dx = \frac{\pi}{2} \sum_{j=1}^{\lfloor (k+1)/2 \rfloor} \binom{k}{2j-1} \frac{1}{j} \left (1-2^{-2j}\binom{2j}{j}\right ).
\end{equation*}
The two sums involved can be evaluated in closed form. First, by integration from $0$ to $1$, the identity 
$$\sum\limits_{j=1}^{\lfloor (k+1)/2 \rfloor}\binom{k}{2j-1} x^{2j-1}=\frac{1}{2}((1+x)^k-(1-x)^k)$$ 
provides
$$\sum_{j=1}^{\lfloor (k+1)/2 \rfloor} \binom{k}{2j-1} \frac{1}{j} = \frac{2^{k+1}-2}{k+1}$$
Second, the known relation (see Bataille \cite{Bat})
\[\sum_{j=0}^{\lfloor{m/2}\rfloor}\binom{m}{j}\binom{m-j}{j}2^{m-2j}=\binom{2m}{m}\quad (m\ge 0),\]
leads to
\begin{align*}
\sum_{j=1}^{\lfloor (k+1)/2 \rfloor} \binom{k}{2j-1} \frac{1}{j} 2^{-2j} \binom{2j}{j} 
&= \frac{2}{k+1}\sum_{j=1}^{\lfloor (k+1)/2 \rfloor} \binom{k+1}{2j}\binom{2j}{j}2^{-2j} \\
&= \frac{2}{k+1}\sum_{j=1}^{\lfloor (k+1)/2 \rfloor} \binom{k+1}{j}\binom{k+1-j}{j}2^{-2j} \\
&= \frac{2}{k+1}\left(2^{-(k+1)}\binom{2k+2}{k+1}-1\right) \\
&= - \frac{2}{k+1} + \frac{2^{-k+1}}{(k+1)^2} (2k+1) \binom{2k}{k}.
\end{align*}
Finally
\begin{equation}\label{arcsin_int}
\int_{-1}^1 (1+x)^k\arcsin(x)\,dx = \pi \left ( \frac{2^k}{k+1} - \frac{2^{-k}}{(k+1)^2} (2k+1) \binom{2k}{k} \right ).
\end{equation}
Using \eqref{arcsin_int} the sum on the left hand side becomes
\begin{align*}
& \sum_{k=0}^n \binom{n+k}{2k} \binom{2k}{k} \frac{(-1)^k}{n+k} 2^{-k} \int_{-1}^1 (1+x)^k\arcsin(x)\,dx \\
&\qquad = \pi \sum_{k=0}^{n} \binom{n+k}{2k} \binom{2k}{k}^2 2^{-2k} (2k+1) \frac{(-1)^{k+1}}{(n+k)(k+1)^2},
\end{align*}
where we have also applied \eqref{id_int1}. The two results in Theorem \ref{arcsin_thm} now follow from the integral
\begin{equation*}
I_n = \int_{-1}^1 P_n(x) \arcsin(x)\,dx = \begin{cases}
0, & n \,\,\text{even} \\ 
 &  \\
\pi \left (\frac{(n-2)!!}{2^{(n+1)/2}\big (\frac{n+1}{2}\big )!}\right )^2, & n \,\,\text{odd};
\end{cases} 
\end{equation*}
which is Eq. 7.249 (1) in Gradshteyn and Ryzhik \cite{GrRy07}.
\end{proof}



\begin{thebibliography}{99}

\bibitem{Aloui}
B. Aloui, Legendre polynomials in terms of integrals involving Hermite polynomials, Period. Math. Hung. 91 (2025), 112--123.

\bibitem{Bat} 
M. Bataille, Solution to 1996 Ukrainian Mathematical Olympiad Problem 8, Crux Math. 27 (7) (2001), 428--429.

\bibitem{Chu}
W. Chu and J. M. Campbell, Expansions over Legendre polynomials and infinite double series identities, Ramanujan J. 60 (2023), 317--353.

\bibitem{Diekemaa}
E. Diekemaa and T. H. Koornwinder, Generalizations of an integral for Legendre polynomials by Persson and Strang,
J. Math. Anal. Appl. 388 (2012), 125--135.

\bibitem{Gautschi}
W. Gautschi, On the preceding paper "A Legendre Polynomial Integral" by James L. Blue, Math. Comp. 146 (1979), 742--743.

\bibitem{GrRy07}
I.~Gradshteyn and I.~Ryzhik, Table of Integrals, Series, and Products, Elsevier Academic Press, 2007.

\bibitem{Guo}
V. J. W. Guo, Some congruences involving powers of Legendre polynomials, Integral Transforms Spec. Funct. 26 (2015), 660--666.

\bibitem{Gould}
H. W. Gould, Combinatorial Identities, Published by the author, Revised edition, 1972.

\bibitem{Klemm}
A. D. Klemm and S. Y. Larsen, Some integrals involving Legendre polynomials providing combinatorial identities, 
J. Austral. Math. Soc. Ser. B 32 (1991), 304--310.

\bibitem{OEIS}
N. J. A. Sloane, \emph{The On-Line Encyclopedia of Integer Sequences}, https://oeis.org.

\bibitem{Srivastava}
H. M. Srivastava and J. Choi, \emph{Series Associated with the Zeta and Related Functions}, Springer Science+Media, B.V., 2001.

\bibitem{Wan}
J. Wan and W. Zudilin, Generating functions of Legendre polynomials: a tribute to Fred Brafman, J. Approx. Theory 164 (2012), 488--503.

\end{thebibliography}
\end{document}